\documentclass{article}

\usepackage{amsmath,amssymb,amscd}
\usepackage{amsthm}

\usepackage{graphicx}
\usepackage{epstopdf}
\usepackage{color}
\usepackage{url}

\numberwithin{equation}{section}

\newtheorem{theorem}{Theorem}[section]
\newtheorem{assumption}{Assumption}[section]
\newtheorem{corollary}{Corollary}[section]
\newtheorem{proposition}{Proposition}[section]

\newtheorem{definition}{Definition}[section]

\newcommand{\mcr}{{\mathcal R}}
\newcommand{\mcd}{{\mathcal D}}
\DeclareMathOperator*{\argmin}{arg\,min}
\newcommand{\R}{\mathbb{R}}

\newcommand{\xa}{x_\alpha}
\newcommand{\xad}{{x_\alpha^\delta}}
\newcommand{\xd}{x^\dagger}
\newcommand{\yd}{y^\delta}

\begin{document}

\title{The Kurdyka-\L{}ojasiewicz inequality as regularity condition}

\author{Daniel Gerth$^\ast$, Stefan Kindermann$^\dag$\\$^\ast$Faculty of Mathematics, Chemnitz University of Technology,\\ 09107 Chemnitz, Germany\\ \texttt{daniel.gerth@mathematik.tu-chemnitz.de}\\
$^\dag$Industrial Mathematics Institute,\\ Johannes Kepler University Linz, 4040 Linz, Austria,\\ \texttt{kindermann@indmath.uni-linz.ac.at}
}

\maketitle              

\begin{abstract}
We show that a Kurdyka-\L{}ojasiewicz (KL)
inequality can be used as regularity condition
for  Tikhonov regularization with linear operators in Banach spaces.
In fact, we prove the equivalence of a KL inequality 
and various known regularity conditions 
(variational inequality, rate conditions, and others) that 
are utilized  for postulating smoothness conditions to obtain convergence rates. Case examples of rate estimates for Tikhonov regularization with source conditions 
or with conditional stability estimate illustrate 
the theoretical result.
\end{abstract}

\section{Introduction}
In the theory of the 
regularization of ill-posed inverse problems, it is well-known that the behavior of regularization methods
essentially depends on the interplay of the forward operator with the true solution. 
Over time, several conditions have been developed that, usually formulated as assumptions,
allow for a more or less precise description of the regularization process. In this paper,
we will connect the set of smoothness conditions discussed in
the recent paper \cite{HKM19} to a Kurdyka-\L{}ojasiewicz (KL) inequality. 
The KL inequality, which we introduce in detail in Section \ref{sec:KL}, has been 
utilized in various branches of mathematics since its discovery in the 1960's. 
Hence, it may open new perspectives to inverse problems. 

Before going into detail, we introduce 
the setting of our paper. We consider operator equations
\begin{equation}\label{eq:problem}
Ax=y
\end{equation}
where $A$ is a bounded linear operator mapping from an infinite-dimensional Banach space $X$ to an
infinite-dimensional Hilbert space $H$. We assume that the range $\mcr(A)$ of $A$ is not closed in
$H$, $\mcr(A)\neq \overline{\mcr(A)}$, such that $A$ is not continuously invertible and hence \eqref{eq:problem} is ill-posed.  
We assume that only noisy data $y^\delta$ is available with $||y-y^\delta||\leq\delta$ for $\delta>0$. Due to the 
ill-posedness of \eqref{eq:problem} and the noisy data, we employ the Tikhonov-type regularization
\begin{equation}\label{eq:functional}
T_\alpha^\delta(x)=\frac{1}{2}||Ax-y^\delta||^2+\alpha J(x)
\end{equation}
to determine a stable approximation to the true solution $x^\dag$ for which $Ax^\dag=y$ holds. In \eqref{eq:functional}, $\alpha>0$ is the regularization parameter and $J:\mcd(J)\subset X\rightarrow\R$ the penalty functional. The minimizer of \eqref{eq:functional} is the regularized solution, i.e.,
\begin{equation}
x_\alpha^\delta=\argmin_{x\in\mcd(A)}T_\alpha^\delta(x).
\end{equation}
By omitting the superscript $\delta$, we denote noise-free data
and variables, i.e.,
\begin{equation}\label{eq:functional_nonoise}
x_\alpha=\argmin_{x\in\mcd(A)}T_\alpha \qquad\text{with} \quad   T_\alpha(x)=\frac{1}{2}||Ax-y||^2+\alpha J(x).
\end{equation}

In order to guarantee existence and stability of the approximations $x_\alpha^\delta$ and $x_\alpha$, respectively, 
we impose the following standard assumptions (see, e.g., \cite{HKM19,buch_bspacereg}) on the penalty functional $J$
throughout the paper:
\begin{assumption}\label{ass:penalty}
The functional $J:X\rightarrow[0,\infty]$ is a proper, convex functional defined on a Banach space $X$, 
which is lower semicontinuous with respect to weak (or weak*) sequential 
convergence. Additionally, we assume that $J$ is a stabilizing (weakly coercive) functional,
i.e., the sublevel sets $[J\leq c]:=\{x\in X:J(x)\leq c\}$ of $J$ are, for all $c\geq 0$, 
weakly (or weakly*) sequentially compact. 
Moreover, we assume that at least one solution $x^\dag$ of \eqref{eq:problem} with finite penalty value $J(x^\dag)<\infty$ exists
and that the subgradient $\partial J(\xd)$ exists. 
\end{assumption} 

With the basic regularization properties covered as consequence of Assumption \ref{ass:penalty}, we move directly to the discussion of convergence rates. In Banach space regularization, the Bregman distance
\[
B_\xi(z,x):=J(x)-J(z)-\langle \xi,x-z\rangle\geq0,\quad x\in X,\quad \xi\in\partial J(z)\subset X^\ast,
\] 
where the subgradient $\xi$ is an element of the subdifferential $\partial J(z)$ of $J$ in the point $z\in X$, has become a popular choice to measure the speed of convergence of the approximate solution to the true solution $x^\dag$. In this paper, we follow the approach of \cite{HKM19} and consider the Bregman distance
\[
B_{\xi_\alpha^\delta}(x_\alpha^\delta,x^\dag)
\]
with subgradient taken at the approximate solutions. Note that the Bregman distance is not symmetric in its arguments. Our task is to find an index function $\varphi$, i.e., a monotonically increasing function $\varphi:[0,\infty)\rightarrow[0,\infty)$
with $\varphi(0) = 0$ that is  continuous (possibly only in a neighborhood of $0$),
such that
\begin{equation}\label{eq:rate_prototype}
B_{\xi_\alpha^\delta}(x_\alpha^\delta,x^\dag)\leq \varphi(\delta).
\end{equation}
It is well-known that no uniform 
function $\varphi$ exists {\em for all} $x^\dag\in X$, and that $\varphi$ 
has to take into account 
the interplay between 
the operator $A$, the solution $x^\dag$, and the 
penalty functional $J$, in combination with an appropriate choice of the regularization 
parameter $\alpha>0$ in \eqref{eq:functional} and \eqref{eq:functional_nonoise}, respectively. Many conditions have been
developed that control this interplay and yield convergence rates \eqref{eq:rate_prototype}.
It is the aim of this paper to show the equivalence of most of the known conditions, and more important, 
we add another equivalent condition in form of the KL-inequality. 


\section{Convergence rate theory for convex Tikhonov regularization} 



For the complete statement 
of our equivalence results, we also need 
Flemming's   distance function \cite{Flemmex,Flemmdiss}:
\[ D(r) := \sup_{x \in X} \left( J(\xd) - J(x) - r \|A x - A \xd\|\right). \]

\begin{theorem}\label{th1}
The following statements are equivalent: 
\begin{itemize}
\item[(a)] ($J$-rate) There is an index function $\Psi_1$ such that
\begin{equation}\label{eq:ass_penaltydiff}
J(x^\dag)-J(x_\alpha)\leq \Psi_1(\alpha)\quad \mathrm{for\ all} \quad\alpha>0.
\end{equation}
\item[(b)] ($T$-rate) There is an index function $\Psi_2$ such that
\begin{equation}\label{eq:ass_tikhdiff}
\frac{1}{\alpha}\left(T_\alpha(x^\dag)-T_\alpha(x_\alpha)\right)\leq \Psi_2(\alpha)\quad \mathrm{for\ all} \quad\alpha>0.
\end{equation}
\item[(c)] (Variational inequality) There is an index function $\Phi_3$ 
such that
\begin{equation}\label{eq:ass_varinequality}
J(x^\dag)-J(x)\leq \Phi_3(\|Ax^\dag-Ax\|)\qquad \mathrm{for\ all} \quad x\in X.
\end{equation}
 \item[(d)] (Distance function) There is an index function $\Psi_4$ such that 
 \begin{equation} 
 D(\tfrac{1}{r}) \leq \Psi_4(r) \qquad \forall r >0.
 \end{equation}
 \item[(e)] (Dual $T$-rate) There exists 
 an index function $\Psi_5$ such that for all $\alpha>0$ a $z \in Y$ exist with $x^* = \partial J(\xd)$
 and 
 \begin{equation} 
 J^*( A^*z) -J(x^*)  - (\xd,  A^* z-x^*)_{X,X^*}   + \alpha  \frac{1}{2}  \|z\|^2 \leq \Psi_5(\alpha), 
 \end{equation} 
\item[(f)] (KL-inequality) There exists a concave index function $\varphi$ such that \\$(\partial\varphi)^{-1}(z) z$ is nonincreasing with 
$\lim_{z\to \infty} (\partial \varphi)^{-1}(z) z = 0$, with 
\begin{equation}\label{eq:ass_klinit}
\left\|\partial\left(\varphi\circ \left(T_\alpha(x^\dag)-T_\alpha(x_\alpha)\right)\right)\right\|\geq \frac{1}{k}.
\end{equation}
\end{itemize}
\end{theorem}
\begin{proof} 
In the proof we provide the formula for converting the various index functions: 
In \cite[Prop. 2.4]{HKM19} the equivalence of (a) and (b) was shown:
\[ \text{(a)} \Rightarrow \text{(b)}: \Psi_2 \leq \Psi_1 \qquad 
\text{(b)} \Rightarrow \text{(a)}: \Psi_1 \leq 2 \Psi_2.\] 
Also in \cite[Prop 3.3]{HKM19} it was shown that 
\[  \text{(c)} \Rightarrow \text{(b)}: \Psi_2(\alpha) \leq \sup_{t>0}\left( \Phi_3(t) - \frac{t^2}{2 \alpha}\right). \] 
It follows that $\Psi_2$ is increasing and by continuity of $\Phi_3$, it can be shown that $\Phi_2(0) = 0$.  
We now show (b) $\Rightarrow$ (c):  
From  \eqref{eq:ass_tikhdiff}, it follows, for all $x$ and all $\alpha$,
\[ J(\xd) - J(\xa) - \frac{1}{2\alpha} \|A \xa- A \xd\|^2 \leq 
 \frac{1}{\alpha} T_\alpha(\xd,A \xd)  -\frac{1}{\alpha} T_\alpha(\xa,A \xd) \leq \Psi_2(\alpha).  \] 
Thus, from the optimality  of $\xa$, we find 
\[ J(\xd) - J(x)  \leq \Psi_2(\alpha) + \frac{1}{2\alpha} \|A x- A \xd\|^2. \]  
Taking the infimum over $\alpha$ yields the variational inequality  \eqref{eq:ass_varinequality}
with the function $\Phi_3$ 
\[  \text{(b)} \Rightarrow \text{(c)}: \Phi_3(\alpha) = \inf_{t>0} \left(\Psi_2(t) + \frac{\alpha^2}{2 t} \right).  \]
If follows easily that $\Psi$ is an index function.

Moreover, (d) $\Leftrightarrow$  (c) by results of Flemming \cite[Lemma 3.4]{Flemmex} \cite[Thm.~12.32]{Flemmdiss}, with 
\[  \text{(d)} \Rightarrow \text{(c)}: \Phi_3(\alpha)  = \inf_{r>0} \left( \Phi_4(r) + r  \alpha \right) \]
and 
\[  \text{(c)} \Rightarrow \text{(d)}: \Phi_4(\alpha)  = \inf_{r>0} \left( \Phi_3(t) -   \alpha t \right). \]

Concerning (f), we remark that by duality we may 
rewrite the Tikhonov functional as 
\begin{align*}
&\tfrac{1}{\alpha}\left(T_\alpha(x^\dag)-T_\alpha(x_\alpha)\right) = 
J(x^\dag) - \tfrac{1}{\alpha}T_\alpha(x_\alpha) \\
&= 
J(x^\dag) -\tfrac{1}{\alpha} \sup_{p}\left[ -\tfrac{1}{2}\|p\|^2  + (p,A \xd)   -\alpha J^*\left(\frac{1}{\alpha} A^* p\right)\right] \\
& =  \inf_{p}\left[ J(x^\dag)+ J^*\left(\frac{1}{\alpha} A^* p\right) - \tfrac{1}{\alpha} (p,A \xd)  + \tfrac{1}{2 \alpha }\|p\|^2\right].
\end{align*}
Young's inequality yields  $J(x^\dag) = (\xd,x^*) -J^*(x^*)$, and by setting $z  = \frac{1}{\alpha} p$ 
it is clear that (f) is just a reformulation of (b): (Note that the infimum over $p$ is attained). 
\[ \text{(b)} \Leftrightarrow \text{(f)}:   \Psi_5 = \Psi_2. \]
Similar formulas were actually already used by Flemming \cite{Flemmdiss}.

The essential equivalence of the 
KL inequality (g) is one of the main issues in this paper and will be shown in later sections in Theorem~\ref{th:main}.
\end{proof} 

Hence, any of the conditions in Theorem~\ref{th1} implies the other ones. 
These conditions imply a certain decay rate for the 
approximation error in the Bregman distance.  This subsequently yields 
convergence rate for the total error measured in the Bregman distance. 
Not only this, but we immediately obtain errors in the strict metric and a 
Tikhonov rate (These results were obtained or follow easily from 
\cite[Thm.~2.8, Prop. 3.7] {HKM19}):

\begin{theorem}
Let any of the equivalent assumptions in Theorem~\ref{th1} hold. 
Then, for all $\alpha >0$, 
\begin{enumerate}
\item (Bregman rate)
there is a constant  $C$ such that 
\begin{equation}\label{eq:bregman_decomposition}
B_{\xi_\alpha^\delta}(x_\alpha^\delta,x^\dag)\leq C \inf_\alpha \left( \frac{\delta^2}{\alpha}+\Psi_2(\alpha) \right);
\end{equation}
\item (strict metric rate) 
there is a constant  $C$ such that 
such that for all $\alpha>0$
\begin{equation}\label{strict metric} 
\begin{split} 
 \left|J(\xd) - J(\xad)\right| & \leq  C\left(\Psi_2(\alpha) +  \frac{\delta^2}{\alpha} \right), \\
 \|A \xad - A \xd\|^2 &\leq C
\left( \alpha \Psi_2(\alpha) + \delta^2 \right);
 \end{split} 
\end{equation} 
\item  (Tikhonov rate) there is a constant $C$ such that 
\begin{equation}\label{Tikhonov rate} 
| J(\xd) - \frac{1}{\alpha} T_{\alpha}^\delta(\xad) | \leq C  \left(\Psi_2(\alpha) +  \frac{\delta^2}{\alpha} \right). 
\end{equation} 
\end{enumerate} 
\end{theorem}

Moreover, defining the companion $\Theta(\alpha)$ as
\begin{equation}\label{eq:companion}
\Theta(\alpha):=\sqrt{\alpha\Psi_2(\alpha)},
\end{equation}
the a-priori choice
\begin{equation}\label{eq:alphaapri}
\alpha_\ast=\alpha_\ast(\delta):=\left(\Theta^2\right)^{-1}\left(\frac{\delta^2}{2}\right)=\Theta^{-1}\left(\frac{\delta}{\sqrt{2}}\right)
\end{equation}
obtained by equilibrating the error decomposition \eqref{eq:bregman_decomposition}
yields the following convergence rate: 
\begin{corollary}\label{col:rates_bregman}
Let any of the equivalent assumptions in Theorem~\ref{th1} hold. Then with the choice \eqref{eq:alphaapri} we 
obtain the convergence rates 
\begin{equation}
B_{\xi_\alpha^\delta}(x_\alpha^\delta,x^\dag)\leq 2\Psi_2\left(\Theta^{-1}\left(\frac{\delta}{\sqrt{2}} \right) \right).
\end{equation}
\end{corollary} 
Note that the same rates holds for the analog error measures in 
\eqref{strict metric}  and \eqref{Tikhonov rate}.


\section{The \L{}ojasiewicz-inequality}\label{sec:KL}
In this section we give a brief overview over the Kurdyka-\L{}ojasiewicz (KL) inequality and some of its implications. A main reason for our interest in this inequality is its broad spectrum of applications in several mathematical disciplines. This may open new interconnections for inverse problems. We start with a short and certainly incomplete overview of the KL inequality.

\L{}ojasiewicz showed that for any real analytic function $f:D(f)\subset\mathbb{R}^n\rightarrow \mathbb{R}$ there is $\theta\in[0,1)$ such that 
\[
\frac{|f(x)-f(\bar x)|^\theta}{\|\nabla f(x) \|}
\]
remains bounded around any critical point $\bar x$, i.e., $\nabla f(\bar x)=0$ \cite{loja,loja2}. Kurdyka \cite{Kurdyka} later generalized the result to $C^1$ functions whose graphs belong to an o-minimal structure. A further generalization to nonsmooth subanalytic functions was given in \cite{Bolte07}. It can also be formulated in (general) Hilbert spaces, see, e.g., \cite{Chillhilbert,HARAUX}, and has applications, for example, in PDE analysis (see, for example, \cite{haraux2,huang,Simon}), neural networks \cite{forti} and complexity theory \cite{nesterov}. First approaches towards inverse problems were made in \cite{GerthKL,gg2017}. In the optimization literature, the KL inequality has emerged as a powerful tool to characterize the convergence properties of iterative algorithms; see, e.g., \cite{absil,Attouch,Bolte,Bolte07,Bot,Frankel2015,gg2017}. 

It is known that the KL inequality immediately yields a measure for the distance between the level-sets of a function, which, under some additional assumptions, directly yields convergence rates for the noise free Tikhonov functional \eqref{eq:functional_nonoise}. To show the generality of the KL inequality, we temporarily consider the problem
\[
f(x)\rightarrow \min_{x\in X}
\] 
where $X$ is a complete metric space with metric $d(x,y)$ and $f:X\rightarrow \R\cup \{\infty\}$ is lower semicontinuous. To formulate the result in this abstract setting, we use the following notation.
\begin{definition}
We denote by
\begin{equation}\label{eq:ls}
[t_1\leq f\leq t_2]:=\{x\in X: t_1\leq f(x)\leq t_2\}
\end{equation}
the level-set of $f$ for 
the levels $t_1\leq t_2$. With slight abuse of notation we write, for fixed $x\in X$, $[f(x)]:=[f=f(x)]$.
Furthermore, for any $x\in X$, the distance of $x$ to a set $S\subset X$ is denoted by
\begin{equation}\label{eq:dist}
dist(x,S):=\inf_{y\in S} d(x,y).
\end{equation}
With this we recall the Hausdorff distance between sets,
\begin{equation}\label{eq:hausdorff}
D(S_1,S_2):=\max\{\sup_{x\in S_1}dist(x,S_2),\sup_{x\in S_2} dist(x,S_1) \}.
\end{equation}
\end{definition}

The KL inequality is directly linked to certain index functions, which we specify below.

\begin{definition}\label{def:index}
A concave function $\varphi:[0,\bar r)\rightarrow \R$ is called desingularizuation function or smooth index function if $\varphi\in C(0,\bar r)\cap C^1(0,\bar r)$, $\varphi(0)=0$, and $\varphi^\prime(x)>0$ for all $x\in(0,\bar r)$. We denote the set of all such $\varphi$ with $\mathcal{K}(0,\bar r)$.
\end{definition}

Now we are ready to cite the main inspiration for our work. It is taken from \cite{BDLM}. In comparison to the original result we have omitted a third equivalence to the concept of metric regularity, see \cite{Ioffe}. Note that we replaced $f$ with $f-\inf f$.
\begin{proposition}{\cite[Corollary 4]{BDLM}}\label{thm:equivalentstuff}
Let $f:X\rightarrow \R\cup \{\infty\}$ be a lower semicontinuous function defined on a complete metric space and $\varphi\in \mathcal{K}(0, r_0)$. Assume that $[\inf f<f<r_0-\inf f]\neq \emptyset$. Then the following assumptions are equivalent.
\begin{itemize}
\item[(a)] For all $r_1,r_2\in (\inf f,r_0)$
\begin{equation}\label{eq:rate_equation}
D([f\leq r_1-\inf f],[f\leq r_2-\inf f])\leq k |\varphi(r_1-\inf f)-\varphi(r_2-\inf f)|.
\end{equation}
\item[(b)] For all $x\in[0<f<r_0]$
\begin{equation}\label{eq:kl}
|\nabla(\varphi\circ (f-\inf f))|(x)\geq \frac{1}{k},
\end{equation}
where $|\nabla f|(x):=\limsup_{\tilde x\rightarrow x}\frac{\max(f(x)-f(\tilde x),0)}{d(x,\tilde x)}$ is the strong slope.
\end{itemize}
\end{proposition}
Now we return to $X$ being a Banach space and consider the Tikhonov functional $f=T_\alpha(x)$. Due to the convexity of  the penalty $J$, we can write Proposition \ref{thm:equivalentstuff} in the following way, where 
\begin{equation}\label{eq:remoteness}
\|\partial f(x)\|_{-}:=\inf_{p\in\partial f(x)}\|p\|_{X^\ast}=\mathrm{dist}(0,\partial f(x))=|\nabla f|(x)
\end{equation}
is the \textit{remoteness} of the subdifferential of $f$ in $x$; see also \cite{AzeCor}.

\begin{corollary}\label{col:kl4inverseproblems}
Let either $A$ be injective or $J$ be strictly convex. Then, for the Tikhonov functional $T_\alpha(x)$ from \eqref{eq:functional_nonoise}, the following are equivalent for a smooth index function $\varphi\in\mathcal{K}(0,\tilde r)$, $x\in [T_\alpha(x_\alpha)\leq T_\alpha(x)\leq \tilde r]$, and $0<k<\infty$.
\begin{itemize}
\item[(a)] \[ \|x-x_\alpha\|\leq k \varphi(T_\alpha(x)-T_\alpha(x_\alpha)),\]
\item[(b)] \begin{equation}\label{eq6}
\varphi^\prime(T_\alpha(x)-T_\alpha(x_\alpha))\, \|\partial T_\alpha(x)\|_{-}\geq\frac{1}{k}.
    \end{equation}
\end{itemize}
\end{corollary}
\begin{proof}
Due to Assumption \ref{ass:penalty} minimizers of $T_\alpha(x)$ exist, and due to the injectivity of $A$ or strict convexity of $J$ the minimizers are unique. Hence it is plain to see from the definition of the Hausdorff-metric \eqref{eq:hausdorff} that \[\|x-x_\alpha\|\leq D([T_\alpha(x)],[T_\alpha(x_\alpha)]),\] and we obtain (a). For (semi)-convex functions, the strong slope coincides with $\|\partial T_\alpha(x)\|_{-}$ (\cite[Remark 12]{BDLM}), from which the remainder follows. \qed
\end{proof}

We close this section by mentioning two obstacles in the application of Corollary \ref{col:kl4inverseproblems}. Firstly, it should be noted that a functional $f=g+h$ is does not necessarily fulfill a KL inequality although both $g$ and $h$ do so. It is therefore not clear how to properly  treat such a sum functional. While a partial answer is given in \cite[Theorem 3.11]{gg2017}, we can not apply the results since they require an invertible operator $A$. We will sketch in Section \ref{sec:example_tikh} that the Tikhonov functional \eqref{eq:functional_nonoise} behaves differently than it would be expected from the sum of its parts.
The second issue in applying Corollary \ref{col:kl4inverseproblems} lies in the fact that it only holds in the noise-free case. To the best of the authors knowledge, there are no results on how the KL inequality behaves under noisy data. It is, however, out of the scope of this paper to close this gap.

\section{The KL-regularity condition}
Due to the equivalences of Theorem~\ref{th1}, it is sufficient to connect one of the conditions (a)-(e) with the KL inequality, and (b) appears to be most simple.  
\begin{theorem}\label{th:main}
The following are equivalent:
\begin{enumerate}
\item[(a)] There is a $\varphi\in\mathcal{K}(0,\infty)$ such that $(\partial\varphi)^{-1}(z) z$ is nonincreasing with\linebreak 
$\lim_{z\to \infty} (\partial\varphi)^{-1}(z) z = 0$
and a constant $k$ such that
\begin{equation}\label{eq:ass_kl}
\|\partial\left(\varphi\circ \left(T_\alpha(x^\dag)-T_\alpha(x_\alpha)\right)\right)\|\geq \frac{1}{k},
\end{equation}
\item[(b)] There is an index function $\Psi$ 
such that
\begin{equation}
\frac{1}{\alpha}\left(T_\alpha(x^\dag)-T_\alpha(x_\alpha)\right)\leq \Psi(\alpha)\quad  \text{for all} \quad\alpha>0.
\end{equation}
\end{enumerate}
The functions $\varphi$ and $\Psi$ are connected via $\Psi(t)=\tfrac{1}{t}(\partial \varphi)^{-1}\left(\frac{1}{t k\|[\partial J](x^\dag)\|_{-}} \right)$. 
\end{theorem}
\begin{proof}
First, we observe that in our context, where $x_\alpha$ is the minimizer of the Tikhonov functional and $x^\dag$ is the point of interest, the KL inequality \eqref{eq:ass_kl} can be written as
\begin{equation}\label{eq:kl_tikhinserted}
\partial \varphi\left( T_\alpha(x^\dag,y)-T_\alpha(x_\alpha,y)\right)\mathrm{dist}(0,[\partial T_\alpha(x^\dag,y)](x^\dag))\geq \frac{1}{k},
\end{equation}
where 
\[
[\partial T_\alpha(x^\dag,y)](x^\dag)=A^\ast(Ax^\dag-y)+\alpha[\partial J](x^\dag)=\alpha[\partial J](x^\dag).
\]
By concavity,  $\partial \varphi$  is monotonically decreasing and thus  \eqref{eq:kl_tikhinserted} leads to 
%
%
%
\[
T_\alpha(x^\dag)-T_\alpha(x_\alpha)\leq {\partial \varphi}^{-1}\left(\frac{1}{k\alpha\|\partial J(x^\dag)\|_{-}}\right).
\]
Dividing both sides by $\alpha>0$ yields (b) with 
\[
\Psi(\alpha)=\tfrac{1}{\alpha}{\partial \varphi}^{-1}\left(\frac{1}{k\alpha\|\partial J(x^\dag)\|_{-}}\right).
\]
This function is an index function by assumptions.

On the other hand, we write (b) as
\[
T_\alpha(x^\dag)-T_\alpha(x_\alpha)\leq \alpha\Psi(\alpha), 
\]
and by  defining
\[
\bar\Theta(\alpha):=\frac{\Psi(\frac{1}{\alpha})}{\alpha}
\]
we have
\[
\Delta T\leq \bar\Theta\left(\frac{1}{\alpha}\right).
\]
As $\Psi(\frac{1}{\alpha})$ is nonincreasing  so is 
$\bar \Theta$, hence 
\[
\bar\Theta^{-1}(\Delta T)\geq \frac{1}{\alpha}.
\]
 Finally, 
identifying $\partial \varphi=\bar\Theta^{-1}$
and noting that  $\|\partial T_\alpha(x^\dag)\| \sim \alpha$, 
we get the KL inequality \eqref{eq:kl_tikhinserted} up to  constants.
As  $\bar\Theta^{-1}$ is nonincreasing, $\varphi$ is concave.
Note that $\partial \varphi^{-1}(z) z = \Psi(\frac{1}{z})$ such that the stated condition 
on $\varphi$ follow as $\Psi$ is an index function. \qed
\end{proof}

It is interesting that in the proof we stumbled upon the companion function $\Theta$ from \eqref{eq:companion}.
Namely, we have $\Theta^2(\alpha)=\bar\Theta(\frac{1}{\alpha})$. The proof also reveals the identification
\[
\Theta^2(\alpha)=(\partial \varphi)^{-1}\left(\frac{c}{\alpha}\right).
\]
Equation \eqref{eq:alphaapri} for the a priori choice ($\Theta^2(\alpha_*) \sim \delta^2$) 
of the regularization parameter then reads
\begin{equation}\label{eq:alphaapri_kl}
\alpha_\ast=\frac{1}{\partial \varphi(\delta^2)},
\end{equation}
and we obtain the formal convergence rate
\[
B_{\xi_\alpha^\delta}(x_\alpha^\delta,x^\dag)\leq 
\partial \varphi(\delta^2) (\partial \varphi)^{-1}\left( c  \partial \varphi(\delta^2)\right) 
\sim \partial \varphi(\delta^2)  \delta^2. \]
%
%
Since $\varphi\in\mathcal{K}(0,r_0)$ is by definition concave, it 
holds that 
\[
\partial \varphi(\delta^2)  \delta^2  \leq \varphi(\delta^2), \]
which follows from the property of the ``subgradient'' of concave functions, 
where the inequality is reversed compared to convex ones: 
\[ \partial \varphi(x)(0-x) + \varphi(x) \geq \varphi(0) = 0. \]

\section{Relation to conditional stability estimates}
We illustrate  how the KL-theory quite directly yields convergence rates 
in case that a conditional stability estimate holds. 
Note that such estimates are a very useful tool 
in, e.g., parameter identification problems 
in partial differential equations;
for examples, see, e.g., 
\cite{BuChYa,ChYa,IY,Ya}.
The use of conditional stability 
estimates \eqref{eq:condstab} for rate estimates 
was in particular 
investigated 
by Cheng and Yamamoto 
in the seminal article \cite{ChengYama}.

Consider the Tikhonov functionals 
\begin{align} \label{eq:tikhonov_chengyam}
T_\alpha (x) &= \frac{1}{2} \| A x - y\|_Y^2 + \alpha \frac{1}{2}\|x\|_Z^2, \\  
T_\alpha^\delta (x) &= \frac{1}{2} \| A x - y^\delta \|_Y^2 + \frac{1}{2}\alpha \|x\|_Z^2,\nonumber 
\end{align} 
where $A: Z\to Y$ and $y = A \xd$.  We furthermore assume that the Hilbert space $Z\hookrightarrow X$ is continuously embedded into a Banach space $X$,
and there we assume a conditional stability estimate
to hold (which, for simplicity, we 
take as a H\"older function): for some 
 $0\leq  \alpha < 2$ we assume that
 \begin{equation}\label{eq:condstab}
 \|f_1 -f_2\|_X \leq \|A f_1 - A f_2 \|_Y^\alpha \qquad \qquad \forall \|f_1\|_Z, \|f_2\|_Z \leq C. 
 \end{equation}
Cheng and Yamamoto have considered precisely  this setup and verified convergence rates. 

Here we  illustrate the approach via the KL-inequality. To this end, we extend the Tikhonov functionals as follows to $X$:
\begin{align*}
\bar{T}_\alpha (x) := \begin{cases} T_\alpha(x) & \text{if } x \in Z, \\ 
\infty  & \text{if } x \in X, x \not\in Z. \end{cases}
\end{align*}

At first we verify the KL-inequality
\eqref{eq6}
for $\bar{T}_\alpha$ on $X$. 
Note that it is enough to consider 
the inequality for  $\bar{T}_\alpha(x) < \infty$, thus for 
$x \in Z$. In this case it reads 
\[ \phi'(T_\alpha(x) - T_\alpha(\xa)) \|\partial \bar{T}_\alpha(x) \|_{-} \geq \frac{1}{k}. \] 
In the following we write $A^*$ for the adjoint of $A$ in the space $Z$. 

By \cite[Prop~3.1]{AzeCor} the strong slope or the 
remoteness can be characterized by 
the directional derivative
 $\bar{T}_\alpha'$,
\begin{equation}\label{eq:super} 
\begin{split}
&\|\partial \bar{T}_\alpha(x) \|_{-}
= \sup_{\bar{T}_{\alpha}(z) < \bar{T}_\alpha(x)} 
\frac{- \bar{T}_\alpha'(x,z-x)}{\|x-z\|_X}, \qquad \text{which simplifies to } \\
&= \sup_{{T}_{\alpha}(z) < {T}_\alpha(x)} \frac{\left(\nabla T_\alpha(x),x-z\right)_{Z}}{\|x-z\|_X},
\end{split}
\end{equation}
where 
$\nabla T_\alpha(x)$ is the usual 
gradient in the space $Z$:
\[ \nabla T_\alpha(x) = 
A^*A(x-\xd) +\alpha x. \]
The optimality condition for $\xa$ reads \[ A^*A \xa + \alpha \xa = A^* A \xd. \] 
After some algebraic manipulation exploiting this identity, we obtain 
\begin{align*} 
T_\alpha(x) - T_\alpha(\xa)&= \frac{1}{2} \|A x - A \xd\|^2 + \frac{1}{2}\alpha \|x\|_Z^2 - 
\left(\frac{1}{2} \|A \xa - A \xd\|^2 + \alpha \frac{1}{2}\|\xa\|_Z^2  \right) \\ 
& =  \frac{1}{2}\|A x - A \xa\|^2 + \alpha \frac{1}{2}\|\xa -x\|_Z^2. 
\end{align*} 
Using the optimality condition and the conditional stability estimate \eqref{eq:condstab}, we have 
using \eqref{eq:super}
\begin{align}\label{eq:kl_tikh_critical_step} 
 &\frac{1}{2}\|A x - A \xa\|^2 + \alpha \frac{1}{2}\|\xa -x\|_Z^2
 = 
  \frac{1}{2}(A^*A  (x -  \xa) + \alpha (x-\xa), x-\xa)_Z  \nonumber \\&= 
    \frac{1}{2}(A^*A  (x -  \xd) + \alpha x, x-\xa)_Z   
 =   \frac{1}{2}(\nabla T_\alpha(x),x-\xa)_Z \nonumber \\ 
&\leq   \frac{1}{2}\|\partial \bar{T}_\alpha(x) \|_{-} \|x-\xa\|_X \leq    \frac{1}{2} \|\partial \bar{T}_\alpha(x) \|_{-}
\|A (x-\xa)\|^\alpha  \\
  &\leq   C \|\partial \bar{T}_\alpha(x) \|_{-} \left(  \frac{1}{2}\|A x - A \xa\|^2 
    + \alpha \frac{1}{2}\|\xa -x\|_Z^2\right)^\frac{\alpha}{2}. \nonumber   \end{align} 
 Thus, 
 \[    \left(\frac{1}{2}\|A x - A \xa\|^2 + \alpha \frac{1}{2}\|\xa -x\|_Z^2 \right)^{1-\frac{\alpha}{2}} \leq C \|\partial \bar{T}_\alpha(x) \|_{-},  \]
and consequently
\[
\left(T_\alpha(x) - T_\alpha(\xa)\right)^{1-\frac{\alpha}{2}} \leq C \| \partial T_\alpha(x) \|_{X'}.\] 
We have thus found a KL-inequality \eqref{eq6} with 
\[ \varphi'(t) := \frac{1}{t^{1-\frac{\alpha}{2}}}, \quad \mbox{ 
i.e., }  \qquad 
 \varphi(t) = t^\frac{\alpha}{2}. \] 
We now apply Proposition to $\bar{T}_\alpha$ (which agrees 
with ${T}_\alpha$ for the relevant arguments) and obtain 
\begin{align*} \|\xad -\xd\|_X &\leq |\varphi\left( T_\alpha(\xad) - T_\alpha(\xa)\right) - \varphi\left( T_\alpha(\xd) - T_\alpha(\xa)\right)| \\
&\leq 
C \varphi\left(| T_\alpha(\xad) - T_\alpha (\xd) |\right), \end{align*}
noting that $\varphi$ is H\"older continuous. 
We have 
\begin{align*}
& T_\alpha(\xad) - T_\alpha (\xd)  = 
 \frac{1}{2}  \| A \xad - y\|^2 + \alpha  \frac{1}{2} \|\xad\|_{Z}^2 - 
  \frac{1}{2} \delta^2 - \alpha \frac{1}{2} \|\xd\|^2  \\ & = 
   \frac{1}{2}  \| A \xad - \yd\|^2 -    (A\xad - y, \yd- y)  +  \frac{1}{2} \delta^2+
 \alpha  \frac{1}{2} \|\xad\|_{Z}^2 -  \frac{1}{2} \delta^2
   - \alpha  \frac{1}{2} \|\xd \|_{Z}^2  \\
   & =    \frac{1}{2}  \| A \xad - \yd\|^2 -    (A\xad - y, \yd- y)  + 
    \alpha  \frac{1}{2} \|\xad\|_{Z}^2  -     \alpha  \frac{1}{2} \|\xd \|_{Z}^2.
\end{align*}
Since 
\[  \frac{1}{2}  \| A \xad - \yd\|^2  + 
    \alpha  \frac{1}{2} \|\xad\|_{Z}^2 \leq  \frac{1}{2} \delta^2 +    \alpha  \frac{1}{2} \|\xd \|_{Z}^2 \]
    and 
\begin{align*}    & (A\xad - y, \yd- y)  \leq   \delta  \|A \xad - y\| \leq 
  \delta ( \|A \xad - \yd\| +\delta)  \\
  &\leq   \frac{1}{2}  \delta^2 + \frac{1}{2} \|A \xad - \yd\|^2 \leq 
  \delta^2 +    \alpha  \frac{1}{2} \|\xd \|_{Z}^2, \end{align*} 
    we obtain that 
    \begin{align*}
& |T_\alpha(\xad) - T_\alpha (\xd)| \leq C (\delta^2 +   \alpha  \frac{1}{2} \|\xd \|_{Z}^2 ). 
\end{align*}
Thus, choosing $\alpha \sim \delta^2$ yields
\[  |T_\alpha(\xad) - T_\alpha (\xd)| \leq C \delta^2 \] 
and hence the convergence rate 
\[ \|\xad -\xd \|_X \leq \phi(C \delta^2) \sim \delta^\alpha . \]
This is the same parameter choice and the same rate as obtained by Cheng and Yamamoto. 

\section{Example: Tikhonov regularization}\label{sec:example_tikh}
Due to the (partial) equivalence of the KL-inequality with the conditions of \cite{HKM19}, their examples apply in our case as long as $\Psi$ is a power function. Therefore, we will not go through all of those examples again, but focus on the most prominent one, which is classical Tikhonov regularization
\begin{equation}\label{eq:tikh_classic}
T_\alpha(x):=\|Ax-y\|^2+\alpha\|x\|^2,
\end{equation}
where $A:X\rightarrow Y$ is a linear operator between Hilbert spaces $X$ and $Y$ and $\|\cdot\|$ denotes the norm in the respective spaces. 

As is well known, the convergence behavior of Tikhonov-regularization \eqref{eq:tikh_classic} depends on the specific solution $x^\dag$, and we employ here source conditions of the type
\begin{equation}\label{eq:sc}
x^\dag=(A^\ast A)^\mu w, \quad \|w\|\leq 1, \quad \mu>0.
\end{equation}
While the treatment of more general source conditions $x^\dag=\phi(A^\ast A)w$ is possible within our framework (see \cite{HKM19}), it shall be sufficient here to treat only the classical setting \eqref{eq:sc}.

We recall from \cite{GerthKL} that the residual $\|Ax-Ax^\dag\|^2$ fulfills a KL inequality with 
\begin{equation}\label{eq:kl_res_scdiff}\varphi(t)\sim t^{\frac{\mu}{2\mu+1}}
\end{equation}
if
\begin{equation}\label{eq:scdiff}
x-x^\dag=(A^\ast A)^\mu w, \quad \|w\|\leq 1, \quad \mu>0,
\end{equation}
i.e., both $x$ and $x^\dag$ lie in the source set \eqref{eq:sc}. This will become important again later. For now we simply apply the theory from \cite{HKM19} in the case $0<\mu<\frac{1}{2}$ and demonstrate that the KL inequality and Corollary \ref{col:rates_bregman} yield convergence in the Bregman distance. Before starting, we summarize some results from \cite[Section 4.1]{HKM19}. Namely, we have for \eqref{eq:tikh_classic} and under \eqref{eq:sc} that
\begin{equation}\label{eq:tikh_penaltydiff_lowmu}
\|x^\dag\|^2-\|x_\alpha\|^2\leq \alpha^{2\mu}
\end{equation}
and
\begin{equation}\label{eq:tikh_resalpha_lowmu}
\|Ax_\alpha-Ax^\dag\|^2\leq \alpha^{2\mu+1}.
\end{equation}
Then we have from \eqref{eq:tikh_penaltydiff_lowmu} and \eqref{eq:tikh_resalpha_lowmu} that 
\[
\Delta T=\alpha(\|x^\dag\|^2-\|x_\alpha\|^2)-\|Ax_\alpha-Ax^\dag\|^2\sim\alpha^{2\mu+1}.
\]
Because $\nabla T_\alpha(x^\dag)=\alpha\|x^\dag\|$, the KL inequality requires
\[
\partial \varphi(\alpha^{2\mu+1})\alpha\geq c,
\]
and it is easy to see that we even have equality for
\begin{equation}\label{eq:kl_tikh_straightforward}
\varphi(t)=t^{\frac{2\mu}{2\mu+1}}
\end{equation}
with derivative
\[
\partial \varphi(t)=c t^{-\frac{1}{2\mu+1}}, \quad (\partial \varphi)^{-1}(t)=t^{-(2\mu+1)}.
\]
This function satisfies the condition in Theorem~\ref{th:main}.
From this, we obtain
\[
\Psi(\alpha)=\frac{(\partial \varphi)^{-1}\left(\frac{1}{\alpha}\right)}{\alpha}=\frac{\alpha^{2\mu+1}}{\alpha}=\alpha^{2\mu}.
\]
This yields, according to \eqref{eq:alphaapri_kl}
\[
\alpha^\ast\sim \frac{1}{\partial \varphi(\delta^{2})}\sim \delta^\frac{2}{2\mu+1},
\]
and the convergence rate is given by
\[
B_{\xi_\alpha^\delta}(x_\alpha^\delta,x^\dag)\sim \delta^{\frac{4\mu}{2\mu+1}}.
\]
Identifying $\|x_\alpha^\delta-x^\dag\|^2=B_{\xi_\alpha^\delta}(x_\alpha^\delta,x^\dag)$, we obtain the well-known rate
\[
\|x_\alpha^\delta-x^\dag\|\sim \delta^{\frac{2\mu}{2\mu+1}}.
\]

Note that Corollary \ref{col:kl4inverseproblems} does not apply directly since it would yield a convergence rate $\|x^\dag-x_\alpha^\dag\|\leq c \delta^{\frac{4\mu}{2\mu+1}}$, which is clearly off the correct rate by a square in the exponent. We will now sketch a likely explanation for this.

Comparing the functionals \eqref{eq:tikh_classic} and \eqref{eq:tikhonov_chengyam}, it appears that similar techniques should lead to a KL inequality. This is indeed the case, and we obtain for the classical Tikhonov functional \eqref{eq:tikh_classic}
\[
T_\alpha(x)-T_\alpha(x_\alpha)=\frac{1}{2}\|A x - A \xa\|^2 + \alpha \frac{1}{2}\|\xa -x\|^2.
\]
We follow the next steps to arrive at the equivalent of \eqref{eq:kl_tikh_critical_step}, which reads
\begin{align}\label{eq:abschaetzung}
\begin{split}
&\frac{1}{2}\|A x - A \xa\|^2 + \alpha \frac{1}{2}\|\xa -x\|^2\leq \frac{1}{2}(\nabla T_\alpha(x),x-\xa)\\
& \qquad \leq   \frac{1}{2}\|\partial \bar{T}_\alpha(x) \|_{-} \|x-\xa\|_X.
\end{split}
\end{align}
The conditional stability estimate \eqref{eq:condstab} no longer holds, but the source condition \eqref{eq:scdiff} yields an alternative. Namely, using the interpolation inequality
\begin{equation}\label{eq:interpol}
\|(A^\ast A)^r x\|\leq \|(A^\ast A )^qx\|^{\frac{r}{q}}\|x\|^{1-\frac{r}{q}},
\end{equation}
for all $q>r\geq 0$, we see that
\[
\|\xa -x\|=\|(A^*A)^\mu w\|\leq \|A(x-\xa)\|^{\frac{2\mu}{2\mu+1}}\|w\|^{\frac{1}{2\mu+1}}.
\]
Inserting this into \eqref{eq:abschaetzung}, and following the argument after \eqref{eq:kl_tikh_critical_step}, we obtain
\[
\left( T_\alpha(x)-T_\alpha(x_\alpha)\right)^{1-\frac{\mu}{2\mu+1}}\leq C \|\partial \bar{T}_\alpha(x) \|_{-},
\]
which yields a KL inequality with $\varphi^\prime(t)\sim t^{\frac{\mu}{2\mu+1}-1}$ or
\begin{equation}\label{eq:tikh_kl_scdiff}
\varphi(t)\sim t^{\frac{\mu}{2\mu+1}}.
\end{equation}
Comparing this with the previous results, we see that we have the same function $\varphi$ as for the residual functional \eqref{eq:kl_res_scdiff}, but this $\varphi$ is only the square root of the function from \eqref{eq:kl_tikh_straightforward} that we derived earlier in this section. Note that $\|\cdot\|^2$ fulfill a KL inequality with $\varphi(t)=\sqrt{t}$. The discrepancy is due to the local character of the KL inequality for ill-posed problems. From the optimality condition of the classical Tikhonov functional \eqref{eq:tikh_classic} it follows that $x_\alpha$ (and $x_\alpha^\delta$, respectively) are always in the range of $A^\ast=(A^\ast A)^\frac{1}{2}$. Therefore, while $x^\dag$ may fulfill the source condition \eqref{eq:sc} for arbitrary $0<\mu<\infty$, the source condition \eqref{eq:scdiff} with $x\in\{x_\alpha,x_\alpha^\delta\}$ only holds for $\mu=\frac{1}{2}$, and we can only apply Corollary \ref{col:kl4inverseproblems} in this case. Indeed, using the well-known a priori choice $\alpha\sim \delta^{\frac{2}{2\mu+1}}=\delta$, we have $T_\alpha(x^\dag)-T_\alpha(x_\alpha^\delta)\sim \delta^2$, which yields via Corollary \ref{col:kl4inverseproblems} with $\varphi(t)$ from \eqref{eq:kl_res_scdiff} with $\mu=\frac{1}{2}$ the convergence rate $\|x_\alpha^\delta-x^\dag\|\sim \sqrt{\delta}$. Therefore, the different index functions $\varphi$ \eqref{eq:kl_tikh_straightforward} and \eqref{eq:kl_res_scdiff} are no contradiction.

\section*{Acknowledgement}
Part of this research was started during a visit of the second author at the Chemnitz University of Technology. 
S.K. would like to thank the Faculty of Mathematics in Chemnitz and especially Bernd Hofmann for their 
great hospitality. D.G. would like to thank Prof. Masahiro Yamamoto for his hospitality during his stay in Tokio, where the author first learned of the KL inequality.

\bibliographystyle{plain}

\end{document}